\documentclass[10pt]{amsart}
\usepackage{amsthm, amsfonts, amssymb, amsmath, amscd}
\usepackage[all]{xy}
\usepackage{lmodern}
\usepackage[margin=3.5cm]{geometry}

\newcommand{\Too}{\longrightarrow}

\newcommand{\onto}{\twoheadrightarrow}

\newcommand{\act}{\curvearrowright}

\newcommand{\C}{\mathbb{C}}
\newcommand{\R}{\mathbb{R}}

\newcommand{\Z}{\mathbb{Z}}

\newcommand{\M}{\mathrm{M}}

\newcommand{\G}{\Gamma}

\newcommand{\Cstar}{\mathrm{C^*}}
\newcommand{\Cred}{\mathrm{C^*_r}}
\newcommand{\B}{\mathcal{B}}
\newcommand{\Comp}{\mathcal{K}}
\newcommand{\LL}{\mathcal{L}}
\newcommand{\hi}{C(\bd \G)\rtimes \G}

\newcommand{\rcross}{\rtimes_\mathrm{r}}

\newcommand{\K}{\mathrm{K}}
\newcommand{\KK}{\mathrm{KK}}

\newcommand{\E}{\mathrm{E}\hspace{.07pc}}
\newcommand{\Ebar}{\overline{\mathrm{E}}\hspace{.07pc}}

\newcommand{\cov}{\mathrm{cov}}
\newcommand{\dev}{\sigma}
\newcommand{\Tr}{\mathrm{Tr}}

\newcommand{\Hyp}{\mathbb{H}} 

\newcommand{\e}{\epsilon}

\newcommand{\la}{\langle}
\newcommand{\ra}{\rangle}
\newcommand{\bd}{\partial}
\newcommand{\Lip}{\mathrm{Lip}}
\newcommand{\D}{\:\mathrm{d}}

\newcommand{\hdim}{\mathrm{hdim}}
\newcommand{\visdim}{\mathrm{visdim}\:}
\newcommand{\topdim}{\mathrm{topdim}\:}
\numberwithin{equation}{section}

\newtheorem{thm}{Theorem}[section]
\newtheorem{lem}[thm]{Lemma}
\newtheorem{cor}[thm]{Corollary}
\newtheorem{prop}[thm]{Proposition}

\theoremstyle{definition}

\newtheorem{defn}[thm]{Definition}
\newtheorem{rem}[thm]{Remark}
\newtheorem{ex}[thm]{Example}

\begin{document}
\title[Fredholm modules for boundary actions of hyperbolic groups]{Finitely summable Fredholm modules for boundary actions of hyperbolic groups}
\author{Heath Emerson}
\address[H.E.]{Department of Mathematics and Statistics, University of Victoria, Victoria (Canada)}
\email{hemerson@uvic.ca}
\author{Bogdan Nica}
\address[B.N.]{Mathematisches Institut, Georg-August Universit\"at G\"ottingen, G\"ottingen (Germany)}
\email{bogdan.nica@gmail.com}
\date{\today}

\begin{abstract}
We construct a family of odd, finitely summable Fredholm modules over the crossed product $\Cstar$-algebra $\hi$ associated to the action of a non-elementary hyperbolic group $\G$ on its Gromov boundary $\bd \G$. These Fredholm modules all represent the same, distinguished class in \(\K\)-homology, namely that of the `boundary extension' of \(\hi\) associated to the Gromov compactification of \(\G\), and is typically nonzero. Their summability is closely related to the Hausdorff dimension of the boundary. We use these results to compute the Connes-Chern character of the boundary extension in cyclic cohomology. \end{abstract}

\maketitle

\section{Introduction}
From the perspective of Noncommutative Geometry, the action of a non-elementary hyperbolic group $\G$ on its Gromov boundary $\bd\G$, which is a compact metrizable space without isolated points, is encoded by the crossed product $\Cstar$-algebra $\hi$. The amenability of the action of $\G$ on $\bd\G$ (due to Adams \cite{Ada}, but see also Germain \cite[Appendix]{AR00} and Kaimanovich \cite{Kai}) means that the full and the reduced crossed products coincide, and $\hi$ is  nuclear (Anantharaman-Delaroche \cite{Ana02}). Laca and Spielberg \cite{LS96}, and Anantharaman-Delaroche \cite{Ana97}, have shown that $\hi$ is purely infinite simple. Much is known about the  K-theory and the K-homology of $\hi$, in which some special features appear -- particularly the Poincar\'e duality result of Emerson \cite{Emerson}, and the Gysin-type sequence obtained by Emerson--Meyer \cite{Emerson:Euler}. Poincar\'e duality for $\hi$ asserts that the K-theory and the K-homology of $\hi$ are canonically isomorphic with a parity shift, $\K_*(\hi)\simeq \K^{*+1}(\hi)$, whenever $\G$ is torsion-free. The map is induced by cup-cap-product with a particular `fundamental class' which is closely related to the boundary extension class, and its Fredholm module representatives studied in this article. 
The Gysin sequence for $\hi$ is an  exact sequence describing the map on K-theory induced by the inclusion $\Cred\G\to\hi$, where $\Cred\G$ is the reduced group $\Cstar$-algebra of $\G$. The name comes from the analogy with the classical Gysin sequence for the inclusion $C(M)\to C(SM)$, where $M$ is a compact manifold and $SM$ is the sphere bundle. And indeed, the analogy between the $\Cstar$-algebra \(\hi\), and sphere bundles of compact manifolds, is a rather strong one, with several aspects.

In this paper, we are interested in the noncommutative geometric aspect of $\hi$. By a result of Connes \cite[Thm.8]{Con89}, we know that a purely infinite simple $\Cstar$-algebra has no finitely summable spectral triples. Therefore \(\hi\) cannot be endowed with a noncommutative Riemannian geometry (this amounts to the fact that \(\bd\G\) has no \(\G\)-invariant metric, nor, actually, any \(\G\)-invariant probability measures -- this is the `Type III' situation). 
However, \(\bd \G\) \emph{does} have an invariant Lipschitz geometry. We are going to exploit this fact to show that $\hi$ does admit finitely summable \emph{Fredholm modules} which are natural and interesting. The informal interpretation is that $\hi$ does not have a finite-dimensional metric geometry, yet it has a finite-dimensional conformal geometry. 

Let us describe the finitely summable Fredholm modules we study herein. The construction relies on suitable probability measures on the boundary $\bd \G$. These are obtained by realizing the boundary $\bd \G$ as a limit set, namely, we let $\G$ act geometrically  -- i.e., isometrically, properly, and cocompactly -- on a Gromov hyperbolic space $X$. In many examples of interest there is a natural choice for \(X\), \emph{e.g.}, for $\G$ a cocompact lattice in $SO(n,1)$ we can take $X=\Hyp^n$; otherwise, $X$ can be manufactured as the Cayley graph of $\G$ with respect to a finite generating set. The boundary as a topological space is a quasi-isometry invariant, so $\bd X$ serves as a concrete topological model for $\bd \G$ for any such choice of $X$. The boundary as a metric-measure space, on the other hand, depends sensitively on the `hyperbolic filling' $X$. On $\bd X$, there is a natural collection of visual metrics. Each visual metric assigns a finite Hausdorff dimension to $\bd X$, and we define the \emph{visual dimension} $\visdim \bd X$ to be the infimum over all the visual Hausdorff dimensions of $\bd X$. The visual dimension is at least as large as the topological dimension. By a visual probability measure on $\bd X$ we mean the normalized Hausdorff measure defined by a visual metric. If we equip $\bd X$ with a visual probability measure $\mu$, then we obtain a faithful representation of $C(\bd \G)$ on $L^2(\bd X, \mu)$ by multiplication, which in turn induces a faithful representation of $\hi$ on $\ell^2(\G,L^2(\bd X,\mu))$ -- the `regular representation' of the crossed-product, corresponding to the measure \(\mu\). Our first result is that this representation of $\hi$, denoted $\lambda_\mu$, together with the projection $P_{\ell^2\G}$ onto the subspace $\ell^2\G$, regarded as constant functions on \(\bd X\), defines an odd, finitely summable Fredholm module for $\hi$: 

\begin{thm}\label{general} Let $\G$ be a non-elementary hyperbolic group acting geometrically on a Gromov hyperbolic space $X$. Endow the boundary $\bd X$ with a visual probability measure $\mu$, and let $\lambda_\mu$ be the regular representation of $\hi$ on $\ell^2(\G,L^2(\bd X,\mu))$. Then $(\lambda_\mu, P_{\ell^2\G})$ is a Fredholm module for $\hi$ which is $p$-summable for every $p>\max\{2,\visdim \bd X\}$.
\end{thm}

In general, there is no intrinsic reason for choosing a specific visual metric on the boundary of a hyperbolic space. Consider, however, the case of $\Hyp^{n}$: its boundary is the sphere $S^{n-1}$, and the preferred visual metric is the spherical metric. The normalized Hausdorff measure corresponding to the spherical metric is the usual spherical measure. Hence Theorem~\ref{general} yields the following:

\begin{cor}\label{classical} Let $\G$ be a cocompact lattice in $SO(n,1)$. Then the regular representation of $C(S^{n-1})\rtimes \G$ on $\ell^2(\G,L^2(S^{n-1}))$, together with the projection on $\ell^2\G$, defines an odd Fredholm module for $C(S^{n-1})\rtimes \G$ which is $p$-summable for every $p>\max\{2, n-1\}$.
\end{cor}
Theorem~\ref{general} is proved in Section~\ref{sec: finite summability}. There is a slightly sharper form of Theorem~\ref{general} in the case when the visual dimension of $X$ is attained. This happens, for instance, in the situation of Corollary~\ref{classical}, which leads to the following improvement in summability: if $n\geq 4$, then the indicated Fredholm module for $C(S^{n-1})\rtimes \G$ is $(n-1)^+$-summable.

Our second result is that the Fredholm modules just defined represent the \emph{boundary extension class} of \(\G\). The boundary extension is the $\Cstar$-algebra extension 
\begin{align*}
0 \longrightarrow C_0(\G)\rtimes \G \simeq \mathcal{K}(\ell^2\G) \longrightarrow 
C(\overline{\G})\rtimes \G \longrightarrow C(\bd \G)\rtimes \G\longrightarrow 0
\end{align*}
where $\overline{\G}$ is the \(\G\)-equivariant compactification of $\G$ by the boundary $\bd\G$.  This extension determines a class in the Brown--Douglas--Fillmore group $\mathrm{Ext}(\hi)$, which, by the nuclearity of \(\hi\), is naturally isomorphic to Kasparov's \(\K^1(\hi) := \KK_1(\hi, \C)\). 

\begin{thm}\label{represent}
Let $\G$ be a non-elementary hyperbolic group acting geometrically on a Gromov hyperbolic space $X$. Endow the boundary $\bd X$ with a visual probability measure $\mu$, and let $\lambda_\mu$ be the regular representation of $\hi$ on $\ell^2(\G,L^2(\bd X,\mu))$. Then the odd Fredholm module $(\lambda_\mu, P_{\ell^2\G})$ represents the boundary extension class $[\bd_\G]\in\K^1(\hi)$.
\end{thm}


Theorem~\ref{represent} is discussed in Section~\ref{sec: boundary extension class}. The non-triviality of the boundary extension class $[\bd_\G]$ in $\K^1(\hi)$ is a point of great interest. For torsion-free $\G$, Poincar\'e duality and the Gysin sequence yield the following fact (see \S\ref{what is the boundary extension}).

\begin{thm}[from \cite{Emerson}, \cite{Emerson:Euler}]\label{tf} Let $\G$ be torsion-free. Then $[\bd_\G]$ is non-zero if and only if $\chi(\G)\neq\pm 1$. 
\end{thm}

If $\G$ is no longer torsion-free, then Poincar\'e duality is not available and the Gysin sequence is far less explicit. Nevertheless, we can partly extend Theorem~\ref{tf} to virtually torsion-free $\G$ by using Theorem~\ref{represent} and the fact that regular Fredholm modules are well-behaved under passage to finite-index subgroups. We get the following sufficient criterion for non-triviality, formulated in terms of the rational Euler characteristic:

\begin{thm}\label{vtf}
Let $\G$ be virtually torsion-free. If $\chi(\G)\notin 1/\Z$, then $[\bd_\G]$ is non-zero.
\end{thm}

In a forthcoming paper \cite{EN2} we describe a framework of `Dirac classes' in 
connection with smooth actions of discrete groups on smooth manifolds, and we show that 
in the case of a classical hyperbolic group acting on the boundary sphere of \(\Hyp^n\), 
the boundary extension class \([\bd_\G]\) discussed here, is the same as the Dirac class 
of the action. This allows us to compute in classical differential-topological terms, the 
map on \(\K\)-theory induced by the boundary extension class, in the case of hyperbolic 
groups of zero Euler characteristic (like discrete co-compact groups of isometries of 
hyperbolic \(3\)-space). 


\section{Regular Fredholm modules for crossed products}\label{paradigm}
In this section, we address the general question of constructing odd Fredholm modules for reduced $\Cstar$-crossed products. Given a crossed product $C(X)\rcross G$, where $G$ is a discrete group acting by homeomorphisms on a compact metrizable space $X$, we consider the regular representations determined by probability measures on $X$. Such a regular representation gives rise to an odd Fredholm module for $C(X)\rcross G$ as soon as the probability measure has a certain dynamical behavior under the action of $G$. Furthermore, the summability of these regular Fredholm modules can be conveniently described by a dynamical notion of standard deviation.

\subsection{Preliminaries on odd Fredholm modules} We start with a brief review of the relevant definitions. Our references are Connes \cite[Ch.4]{Connes} and Higson - Roe \cite[Ch.8]{HR}. 

Let $A$ be a unital $\Cstar$-algebra. An \emph{odd Fredholm module} for $A$ is a pair $(\pi,P)$, consisting of 

\begin{itemize}
\item[$\bullet$] a representation $\pi:A\to \mathcal{B}(H)$ of $A$ on a Hilbert space $H$, and 

\item[$\bullet$] a projection $P$ in $\mathcal{B}(H)$, 
\end{itemize}
\noindent such that the commutator $[P, \pi(a)]=P\pi(a)-\pi(a)P$ is compact for all $a\in A$. 

Let $p\geq 1$. An odd Fredholm module $(\pi,P)$ is \emph{$p$-summable}, respectively \emph{$p^+$-summable}, if $[P, \pi(a)]$ is in $\mathcal{L}^p(H)$, respectively in $\mathcal{L}^{p+}(H)$, for all $a$ in a dense subalgebra of $A$. Recall, the $L^p$-ideal $\mathcal{L}^p(H)$ and the weak $L^p$-ideal $\mathcal{L}^{p+}(H)$ are defined as follows. For a compact operator $T\in\mathcal{K}(H)$, let $\{\mu_n(T)\}_{n\geq 1}$ denote the sequence of eigenvalues of $|T|$, arranged in non-increasing order and repeated according to their multiplicity. The compactness of $T$ means that $\mu_n(T)\to 0$. Put $\mathcal{L}^p(H)=\big\{T\in \mathcal{K}(H):\; \sum \mu_n(T)^p<\infty\big\}$, and $\mathcal{L}^{p+}(H)=\big\{T\in \mathcal{K}(H):\; \mu_n(T)=O(n^{-\frac{1}{p}})\big\}$ for $p>1$. (The definition of $\mathcal{L}^{1+}(H)$ is slightly different, and it will not be used in this paper.) We have $\mathcal{L}^p(H)\subset\mathcal{L}^{p+}(H)\subset\mathcal{L}^q(H)$ for all $q>p$. For Fredholm modules, this means that $p$-summability implies $p^+$-summability, which in turn implies $q$-summability for all $q>p$; the latter property can be thought of as ``$p^{++}$-summability''.

The odd Fredholm modules for $A$ are the cycles in Kasparov's odd K-homology group $\K^1(A)$. A K-homology class identifies odd Fredholm modules up to \emph{unitary equivalence}, \emph{operator homotopy}, and \emph{addition of degenerates}. It is rather obvious what the first equivalence relation means. The second stipulates that two Fredholm modules $(\pi,P_0)$ and $(\pi,P_1)$ are equivalent if there is a norm-continuous path of projections $(P_t)_{t\in [0,1]}$ such that $(\pi,P_t)$ is a Fredholm module at all times $t\in [0,1]$. (It is worth emphasizing that the representation $\pi$ is fixed throughout an operator homotopy.) A degenerate Fredholm module is one for which the projection commutes with the representation. Under direct summation of K-homology classes, $\K^1(A)$ is an abelian group. 

\subsection{Regular Fredholm modules I} Let $G$ be a discrete group acting (by homeomorphisms) on a compact metrizable space $X$. The regular representations of the reduced crossed product $C(X)\rcross G$ arise as follows. Let $\mu$ be a Borel probability measure on $X$ with full support, meaning that no non-empty open set in $X$ is $\mu$-negligible. The faithful representation of $C(X)$ on $L^2(X,\mu)$ by multiplication induces a faithful representation of the reduced crossed product $C(X)\rcross G$ on $\ell^2(G, L^2(X,\mu))$. This latter representation, denoted $\lambda_\mu$, is the regular representation of the action $G\curvearrowright X$ with respect to $\mu$. Concretely, $\lambda_{\mu}$ is given as follows:
\begin{align*}
\lambda_{\mu}(\phi)\Big(\sum \psi_h \delta_h\Big)=\sum (h^{-1}.\phi)\psi_h \delta_h, \qquad \lambda_{\mu}(g)\Big(\sum \psi_h \delta_h\Big)=\sum \psi_h \delta_{gh}
\end{align*}
where $\phi\in C(X)$, $g\in G$, and $\sum \psi_h \delta_h\in \ell^2(G, L^2(X,\mu))$.

Thinking of $\ell^2 G$ as the constant-coefficient subspace of $\ell^2(G, L^2(X,\mu))$, we may take the corresponding projection $P_{\ell^2 G}$; is given by coefficient-wise integration: 
\[P_{\ell^2 G}\Big(\sum \psi_h \delta_h\Big)=\sum \bigg(\int \psi_h \D\mu\bigg) \delta_h.\]

We are interested in the event that $(\lambda_\mu, P_{\ell^2 G})$ is a Fredholm module -- or, even better, a summable one -- for $C(X)\rcross G$. To that end, we define dynamical versions of two standard probabilistic notions. On the probability space $(X,\mu)$, momentarily devoid of the $G$-action, there are two important numerical characteristics attached to a continuous functions on $X$: the expectation and the standard deviation. Specifically, for $\phi\in C(X)$ we put
\begin{align*}\E(\phi)=\int \phi \D\mu,\qquad \dev (\phi)=\sqrt{\E\big(|\phi|^2\big)-\big|\E(\phi)\big|^2}.
\end{align*}
Recalling now the $G$-action, we consider the following dynamical counterparts.
\begin{defn}
The \emph{$G$-expectation} and the \emph{$G$-deviation} of $\phi\in C(X)$ with respect to $\mu$ are the functions $\E_G(\phi):G\to\C$ and $\dev_G(\phi):G\to [0,\infty)$ given as follows:
\begin{align*}
\E_G(\phi)(g)=\int g^{-1}.\phi \D\mu=\int \phi \D g_*\mu,\qquad \dev_G (\phi)=\sqrt{\E_G\big(|\phi|^2\big)-\big|\E_G(\phi)\big|^2}
\end{align*}
The notations $\E_G^\mu$ and $\dev_G^\mu$ are occasionally used in order to emphasize the dependence on the probability measure $\mu$.
\end{defn}
An explicit, and useful, formula for the $G$-deviation is
\begin{align}\label{iint}\dev_G(\phi)(g)=\sqrt{\tfrac{1}{2}\iint |\phi(gx)-\phi(gy)|^2 \:d\mu(x)d\mu(y)}.
\end{align}
The Fredholmness and the summability of $(\lambda_\mu, P_{\ell^2 G})$ turn out to be very closely related to the decay of the $G$-deviation $\dev_G$. Such decay properties are features of the probability measure $\mu$, since the underlying topological dynamics $G\act X$ is fixed. Let us first make the following convenient definition.

\begin{defn}
We say that $\mu$ has \emph{$C_0$-deviation} if $\dev_G(\phi)\in C_0(G)$ for all $\phi\in C(X)$, respectively \emph{$\ell^p$-deviation} if $\dev_G(\phi)\in \ell^p G$ for all $\phi$ in a dense subalgebra of $C(X)$. 
\end{defn}

We now have:

\begin{prop}\label{from deviation to Fredholm} $(\lambda_\mu, P_{\ell^2 G})$ is
\begin{itemize}
\item[$\bullet$] a Fredholm module for $C(X)\rcross G$ if and only if $\mu$ has $C_0$-deviation;
\item[$\bullet$] a $p$-summable Fredholm module for $C(X)\rcross G$ provided that $\mu$ has $\ell^p$-deviation. For $p\geq 2$ the converse holds.
\end{itemize}
\end{prop}

\begin{proof} On the one hand, $P_{\ell^2G}$ commutes with the group restriction $\lambda_{\mu}|_G$. On the other hand, $P_{\ell^2G}$ compresses the space restriction $\lambda_{\mu}|_{C(X)}$ to multiplication by the $G$-expectation on $\ell^2G$, that is, for all $\phi\in C(X)$ we have
\begin{align}\label{projection compressed lambda}
P_{\ell^2G}\lambda_{\mu}(\phi)P_{\ell^2G}=\M(\E_G\phi).
\end{align}
For $\phi\in C(X)$, put
\[\Pi(\phi):=(1-P_{\ell^2G})\lambda_{\mu}(\phi^*)P_{\ell^2G}=-(1-P_{\ell^2G})[P_{\ell^2G}, \lambda_{\mu}(\phi^*)].\]
With this notation we have
\[[P_{\ell^2G}, \lambda_\mu(\phi)]=P_{\ell^2G}\lambda_{\mu}(\phi)(1-P_{\ell^2G})-(1-P_{\ell^2G})\lambda_{\mu}(\phi)P_{\ell^2G}=\Pi(\phi)^*-\Pi(\phi^*)\]
and $|\Pi(\phi)|=\M(\sigma_G\phi)$, as
\begin{align*}
\Pi(\phi)^*\Pi(\phi)&=P_{\ell^2G} \lambda_{\mu}(\phi)(1-P_{\ell^2G}) \lambda_{\mu}(\phi^*)P_{\ell^2G}\\
&=P_{\ell^2G} \lambda_{\mu}(\phi\phi^*)P_{\ell^2G}-\big(P_{\ell^2G} \lambda_{\mu}(\phi)P_{\ell^2G}\big)\big(P_{\ell^2G} \lambda_{\mu}(\phi^*)P_{\ell^2G}\big)\\
&=\M\big(\E_G(\phi\phi^*)-(\E_G\phi)(\E_G\phi^*)\big)=\M(\dev_G\phi)^2.
\end{align*}
The first point of the proposition is the following chain of equivalences: $(\lambda_\mu, P_{\ell^2 G})$ is a Fredholm module $\Leftrightarrow$ $[P_{\ell^2G}, \lambda_{\mu}(a)]$ is compact for all $a$ in the algebraic crossed product $C(X)\rtimes_\mathrm{alg} G$ $\Leftrightarrow$ $[P_{\ell^2G}, \lambda_{\mu}(\phi)]$ is compact for all $\phi\in C(X)$ $\Leftrightarrow$ $\Pi(\phi)$ is compact for all $\phi\in C(X)$ $\Leftrightarrow$ $|\Pi(\phi)|=\M(\sigma_G\phi)$ is compact for all $\phi\in C(X)$ $\Leftrightarrow$ $\dev_G \phi\in C_0(G)$ for all $\phi\in C(X)$. 

As for the second point, assume that $\mu$ has $\ell^p$-deviation. Then there is a $G$-invariant, dense $*$-subalgebra $A(X)\subseteq C(X)$ such that $\dev_G \phi\in \ell^p G$ for all $\phi\in A(X)$. Thus $|\Pi(\phi)|=\M(\sigma_G\phi)$ is a $p$-summable operator for all $\phi\in A(X)$; as above, we deduce that $[P_{\ell^2G}, \lambda_{\mu}(a)]$ is a $p$-summable operator for all $a\in A(X)\rtimes_\mathrm{alg} G$. Therefore $(\lambda_\mu, P_{\ell^2 G})$ is $p$-summable.

In what concerns the converse, let us start by showing that
\begin{align}\label{conditional expectation smaller}
\|\Pi(a)\delta_h\|_2\geq \sigma_G\big(\mathbb{E}(a)\big)(h)
\end{align}
for all $h\in G$ and $a\in C(X)\rcross G$. Here $\Pi(a)$ is the obvious generalization of our previous $\Pi(\phi)$, and $\mathbb{E}: C(X)\rcross G\onto C(X)$ denotes the standard conditional expectation. Using along the way the fact that $\Pi(\phi_{g'})^*\Pi(\phi_g)$ is a multiplication operator on $\ell^2G$, we have:
\begin{align*}
\la \Pi(a) \delta_h,\Pi(a)\delta_h\ra&=\sum_{g,g'}  \la \Pi(\phi_g) \delta_{gh},\Pi(\phi_{g'})\delta_{g'h}\ra=\sum_{g,g'}  \la \Pi(\phi_{g'})^*\Pi(\phi_g) \delta_{gh},\delta_{g'h}\ra\\
&=\sum_{g}  \la \Pi(\phi_{g})^*\Pi(\phi_g) \delta_{gh},\delta_{gh}\ra=\sum_{g}  \la \M(\dev_G\phi_g)^2 \delta_{gh},\delta_{gh}\ra\\
&=\sum_{g}  (\dev_G\phi_g)^2(gh)\geq (\dev_G\phi_1)^2(h)=\Big(\sigma_G\big(\mathbb{E}(a)\big)(h)\Big)^2
\end{align*}
Now assume that $(\lambda_\mu, P_{\ell^2 G})$ is a $p$-summable Fredholm module for $C(X)\rcross G$. Thus $[P_{\ell^2G}, \lambda_{\mu}(a)]$, equivalently $\Pi(a)$, is a $p$-summable operator for all $a$ in a dense subalgebra $\mathcal{A}$ of $C(X)\rcross G$. For $p\geq 2$, the $p$-summability of $\Pi(a)$ implies the $p$-summability of $\{\|\Pi(a)\xi_\iota\|_2\}_{\iota\in I}$ for any orthonormal system $(\xi_\iota)_{\iota\in I}$ (\cite[Thm.1.18]{Sim}). In particular $\{\|\Pi(a)\delta_h\|_2\}_{h\in G}$ is $p$-summable, so $\sigma_G\big(\mathbb{E}(a)\big)\in \ell^p G$ by \eqref{conditional expectation smaller}. Thus, we have shown that $\dev_G \phi\in \ell^p G$ for all $\phi\in \mathbb{E}(\mathcal{A})$. It follows that $\{\phi\in C(X):\dev_G \phi\in \ell^p G \}$, which is always a subalgebra of $C(X)$, is dense. We conclude that $\mu$ has $\ell^p$-deviation.
\end{proof}

The $C_0$-deviation of the probability measure $\mu$, which characterizes the Fredholmness of $(\lambda_\mu, P_{\ell^2G})$, admits in turn an appealing measure-theoretic interpretation. Let $\mathrm{Prob}(X)$ denote the space of Borel probability measures on $X$, and equip $\mathrm{Prob}(X)$ with the weak$^*$ convergence induced by $C(X)$: by definition, $\nu_\iota\to \nu$ if $\int \phi\D\nu_\iota\to \int \phi\D\nu$ for all $\phi\in C(X)$. Then $\mathrm{Prob}(X)$ is compact. In particular, push-forwards of $\mu$ by elements of $G$ must accumulate. 

\begin{prop}\label{Furstenberg type condition}
The probability measure $\mu$ has $C_0$-deviation if and only if $g_*\mu$ only accumulates to point measures in $\mathrm{Prob}(X)$ as $g\to\infty$ in $G$.
\end{prop}

\begin{proof}
Assume that $\mu$ has $C_0$-deviation, and let $\nu\in \mathrm{Prob}(X)$ be the limit of a sequence $(g_n)_*\mu$ with $g_n\to \infty$ in $G$. For each $\phi\in C(X)$ we have, on the one hand, that $\dev_G\phi(g_n)$ converges to $0$, and on the other hand that $\dev_G\phi(g_n)$ converges to the standard deviation of $\phi$ with respect to $\nu$. Therefore, $\int |\phi|^2\D\nu=|\int \phi \D\nu|^2$ for all $\phi\in C(X)$. This continues to hold throughout $L^2(X,\nu)$, by the density of $C(X)$ in $L^2(X,\nu)$; indeed, Borel probability measures on compact metrizable spaces are automatically Radon. Taking characteristic functions of measurable sets, we see that $\nu$ is $\{0,1\}$-valued. But the only $\{0,1\}$-valued Borel probability measures on $X$ are the point measures: choosing a compatible metric on $X$, there exists a sequence of full-measure balls with radius converging to $0$, hence a point having full measure.

The converse implication is left to the reader.
\end{proof}

\begin{rem}
The measure-theoretic behavior described in the previous proposition bears some resemblance to the notion of Furstenberg boundary, which we now recall. A compact $G$-space $X$ is said to be \emph{strongly proximal} if, for every $\nu\in \mathrm{Prob}(X)$, the closure of $G\nu\subseteq \mathrm{Prob}(X)$ contains some point measure. A compact $G$-space $X$ is a \emph{boundary in the sense of Furstenberg} if it is minimal and strongly proximal, that is to say, the closure of every $G$-orbit in $\mathrm{Prob}(X)$ contains all the point measures. In our situation, we are requiring a pure-point boundary from a single orbit. Namely, the closure of the orbit $G\mu$ can only contain point measures besides $G\mu$ itself; if $X$ is minimal, then the closure of the orbit $G\mu$ consists of $G\mu$ and all the point measures. 
\end{rem}

\begin{rem} It is readily checked that $(\lambda_\mu, P_{\ell^2G})$ is a degenerate Fredholm module if and only if $\mu$ is a point measure; as $\mu$ is assumed to have full support, this happens if and only if $X$ is a singleton. 
\end{rem}

We now address one aspect of the condition $p\geq 2$, encountered in Proposition~\ref{from deviation to Fredholm}. Namely, we show that double ergodicity of $\mu$ is an obstruction to having $\ell^p$-deviation with $p\leq 2$.

\begin{prop}\label{2ergo}
Assume that all points in $X$ are $\mu$-negligible, and that $\mu\times \mu$ is ergodic for the diagonal action of $G$ on $X\times X$. Then the only functions $\phi\in C(X)$ with $\sigma_G\phi\in \ell^2 G$ are the constant functions. In particular, if $\mu$ has $\ell^p$-deviation then $p>2$.
\end{prop}

\begin{proof}
Note that $X$ is not a singleton, and it has no isolated points since $\mu$ has full support.

Assume that $\phi\in C(X)$ is a non-constant function with the property that $\sigma_G\phi\in \ell^2 G$. By \eqref{iint}, we have
\begin{align*}
\|\dev_G \phi\|^2_{\ell^2G}=\tfrac{1}{2}\iint \sum_{g\in G}|\phi(gx)-\phi(gy)|^2 \:d\mu(x)d\mu(y)
\end{align*}
Therefore $S(x,y)=\sum_{g\in G}|\phi(gx)-\phi(gy)|^2$ defines a $G$-invariant $L^2$ map on $X\times X$. By ergodicity, $S$ is a.e. constant, say $S(x,y)=C$ for almost all $(x,y)\in X\times X$. 

There exists $c>0$ such that the open subset $V=\{(x,y): |\phi(x)-\phi(y)|> c\}\subseteq X\times X$ is non-empty. As $X\times X$ has no isolated points, for each positive integer $N$ there exist disjoint, non-empty open subsets $U_1,\dots, U_N\subseteq V$. Using again the ergodicity assumption, we have that each $G\cdot U_i=\cup_{g\in G}\: gU_i$ is either negligible or of full measure. Since non-empty open subsets of $X\times X$ have positive measure, the latter alternative must occur. It follows that $\cap_{i=1}^N\: G\cdot U_i$ has full measure. Let $(x,y)$ in $\cap_{i=1}^N\: G\cdot U_i$ with $S(x,y)=C$. Thus, for each $i$ we have some $g_i\in G$ such that $(g_ix,g_iy)\in U_i$. Now the $g_i$'s are distinct since the $U_i$'s are disjoint, and $|\phi(g_ix)-\phi(g_iy)|>c$ since $U_i\subseteq V$, so
\[C=S(x,y)\geq \sum_{i=1}^N|\phi(g_ix)-\phi(g_iy)|^2>Nc^2.\]
As $N$ is arbitrary, this is a contradiction.
\end{proof}

\subsection{Regular Fredholm modules II}\label{compinv} Next, we investigate the behavior of our regular Fredholm modules with respect to two alterations, one on the space and the other on the group: changing the Borel probability measure on $X$, respectively passing to a finite-index subgroup of $G$.

Two Borel probability measures $\mu$ and $\mu'$ on $X$ are said to be \emph{comparable} if $\mu\asymp \mu'$, i.e., $C_1\mu\leq \mu'\leq C_2\mu$ for some constants $C_2\geq C_1>0$. Clearly, comparability is an equivalence relation which is finer than the usual equivalence of measures (recall, the latter means that each measure is absolutely continuous with respect to the other).

\begin{prop}\label{deviation for comparable measures}
Let $\mu$ and $\mu'$ be comparable probability measures on $X$ with full support. Then $\mu$ has $C_0$-deviation, respectively $\ell^p$-deviation, if and only if $\mu'$ has $C_0$-deviation, respectively $\ell^p$-deviation.
\end{prop}

\begin{proof}
By formula \eqref{iint}, comparable measures have comparable $G$-deviations.
\end{proof}

\begin{prop}\label{independence of comparable measures}
Let $\mu$ and $\mu'$ be comparable probability measures on $X$ with full support, and assume that they have $C_0$-deviation. Then the Fredholm modules $(\lambda_\mu, P_{\ell^2 G})$ and $(\lambda_{\mu'}, P_{\ell^2 G})$ are K-homologous.
\end{prop}

\begin{proof}
Let $\rho=d\mu'/d\mu$ be the Radon-Nikodym derivative, so $\rho$ is essentially bounded from above and from below by the comparability constants of $\mu$ and $\mu'$. First, we have a unitary
\begin{align*}
U: \ell^2(G, L^2(X, \mu')) \to \ell^2(G, L^2(X, \mu)), \qquad \sum \psi_h\delta_h\mapsto \sum \sqrt{\rho}\: \psi_h\delta_h
\end{align*}
which intertwines the corresponding regular representations of $C(X)\rcross G$, that is, $U \lambda_{\mu'}U^*=\lambda_{\mu}$. We may therefore exchange $(\lambda_{\mu'}, P'_{\ell^2 G})$ for $(\lambda_\mu, UP'_{\ell^2G}U^*)$, where the notation $P'_{\ell^2G}$ is used in order to emphasize the dependence on $\mu'$. We now claim that the Fredholm modules $(\lambda_\mu, UP'_{\ell^2G}U^*)$ and $(\lambda_\mu, P_{\ell^2G})$ are operator homotopic. Note that
\begin{align*}
UP'_{\ell^2G}U^*\big(\sum \psi_h\delta_h\big)= \sum \sqrt{\rho}\: \bigg(\int \sqrt{\rho}\:\psi_h \D\mu\bigg)\delta_h,
\end{align*}
and that $\sqrt{\rho}\in L^\infty (X,\mu)$ with $\|\sqrt{\rho}\|_{L^2(X,\mu)}=1$. For $\eta\in L^\infty (X,\mu)$ satisfying $\|\eta\|_{L^2(X,\mu)}=1$, let
\begin{align*}
&M(\eta): \ell^2(G, L^2(X, \mu)) \to \ell^2(G, L^2(X, \mu)), \qquad \sum \psi_h\delta_h\mapsto \sum \eta\: \psi_h\delta_h;\\
&P(\eta)= M(\bar{\eta})P_{\ell^2G}M(\eta),\qquad \sum \psi_h\delta_h\mapsto \sum \bar{\eta}\: \bigg(\int \eta\:\psi_h \D\mu\bigg)\delta_h.
 \end{align*}
Then $P(\eta)$ is a projection, namely the projection of $\ell^2(G, L^2(X, \mu))$ onto $M(\bar{\eta})\ell^2G$. We also have $[P(\eta),\lambda_{\mu}]=M(\bar{\eta})[P_{\ell^2G},\lambda_{\mu}]M(\eta)$ since $M(\eta)$ and $M(\bar{\eta})$ commute with $\lambda_{\mu}$, so $(\lambda_\mu, P(\eta))$ is a Fredholm module. On the other hand, we have $\|P(\eta_1)-P(\eta_2)\|\leq 2\|\eta_1-\eta_2\|_{L^2(X,\mu)}$; this follows from the fact that
\begin{align*}
\bigg\|\bar{\eta}_1\int \eta_1 \psi\D\mu - \bar{\eta}_2\int \eta_2 \psi\D\mu\bigg\|_2\leq 2\|\eta_1-\eta_2\|_2 \|\psi\|_2
\end{align*}
for all $\psi\in L^2(X, \mu)$. Now let $\eta(t)=(\cos t)\: 1+(i \sin t)\: \sqrt{\rho}$, where $0\leq t\leq\pi/2$. Then $\eta(t)\in L^\infty (X,\mu)$, and $\eta(t)$ describes a continuous path in the unit sphere of $L^2(X,\mu)$ between the constant function $1$ and $i\sqrt{\rho}$. Consequently, $P(\eta(t))$ describes a norm-continuous path between $P(1)=P_{\ell^2G}$ and $P(i\sqrt{\rho})=P(\sqrt{\rho})=UP'_{\ell^2G}U^*$.
\end{proof}

Let $H$ be a subgroup of $G$. Restriction of representations from $C(X)\rcross G$ to $C(X)\rcross H$ takes Fredholm modules for $C(X)\rcross G$ to Fredholm modules for $C(X)\rcross H$, and defines a natural homomorphism of abelian groups
\[\mathrm{res}: \K^1(C(X)\rcross G)\to \K^1(C(X)\rcross H).\]
If $\mu$ has $C_0$-deviation, then $(\lambda^G_\mu, P_{\ell^2G})$ is a regular Fredholm module for $C(X)\rcross G$ whose restriction is a Fredholm module for $C(X)\rcross H$. On the other hand, we can also form the regular Fredholm module $(\lambda^H_\mu, P_{\ell^2H})$ for $C(X)\rcross H$. The homological relation between these two Fredholm modules for $C(X)\rcross H$ is particularly simple in the case when $H$ has finite index in $G$.

\begin{prop}\label{finite index} Assume that $\mu$ has $C_0$-deviation, and that $\{g_*\mu\}_{g\in G}$ forms a family of mutually comparable measures. If $H$ is a finite-index subgroup of $G$, then
\begin{align*}
\mathrm{res}\:  \big[(\lambda^G_\mu, P_{\ell^2G})\big] =[G:H] \: \big[(\lambda^H_\mu, P_{\ell^2H})\big]
\end{align*}
in $\K^1(C(X)\rcross H)$.
\end{prop}

\begin{proof} Put $n=[G:H]$, and pick a transversal $t_1,\dots,t_n$ for the right $H$-cosets. The coset decomposition $\ell^2(G, L^2(X, \mu))=\oplus_{1}^n\: \ell^2(H t_i, L^2(X, \mu))$ yields
\begin{align*}
\mathrm{res}\:  \big[(\lambda^G_\mu, P_{\ell^2G})\big]=\oplus_{1}^n \big[(\lambda_{t_i}, P_{\ell^2(H t_i)})\big]
\end{align*}
in $\K^1(C(X)\rcross H)$, where $\lambda_{t_i}$ denotes the representation of $C(X)\rcross H$ on $\ell^2(H t_i, L^2(X, \mu))$. Now consider $(\lambda_t, P_{\ell^2(H t)})$ for $t\in \{t_1,\dots, t_n\}$. The unitary 
\begin{align*}
R_t: \ell^2(H, L^2(X, \mu))\to \ell^2(H t, L^2(X, \mu)), \qquad \sum \psi_h\delta_h\mapsto \sum \psi_h\delta_{ht}
\end{align*}
implements an equivalence between $(\lambda_t, P_{\ell^2(H t)})$ and $(R_t^*\lambda_t R_t, P_{\ell^2H})$. The representation $R_t^*\lambda_t R_t$ on $\ell^2(H, L^2(X,\mu))$ is given by
\begin{align*}
R_t^*\lambda_t R_t(\phi)\Big(\sum \psi_h \delta_h\Big)=\sum t^{-1}.(h^{-1}.\phi)\psi_h \delta_h, \qquad R_t^*\lambda_t R_t(h')\Big(\sum \psi_h \delta_h\Big)=\sum \psi_h \delta_{h'h}
\end{align*}
for $\phi\in C(X)$ and $h'\in H$. Next, the unitary 
\begin{align*}
V_t: \ell^2(H, L^2(X, \mu))\to \ell^2(H, L^2(X, t_*\mu)), \qquad \sum \psi_h\delta_h\mapsto \sum (t.\psi_h)\delta_h
\end{align*}
makes $(R_t^*\lambda_t R_t,  P_{\ell^2H})$ and $(\lambda^H_{t_*\mu},  P_{\ell^2H})$ equivalent. On the other hand, the assumption that $\{g_*\mu\}_{g\in G}$ consists of mutually comparable measures implies, in light of Proposition~\ref{independence of comparable measures}, that $(\lambda^H_{t_*\mu},  P_{\ell^2H})$ and $(\lambda^H_{\mu},  P_{\ell^2H})$ are homologous. Summarizing, we have
\begin{align*}
\big[(\mathrm{res}(\lambda^G_\mu), P_{\ell^2G})\big]=\oplus_{1}^n\: \big[(\lambda^H_\mu, P_{\ell^2H})\big]\end{align*}
in $\K^1(C(X)\rcross H)$, as desired.
\end{proof}

\subsection{Connes' Chern character of a regular Fredholm module}
Assume that $\mu$ has $\ell^p$-deviation. Let $A_p(X)=\{\phi\in C(X):\dev_G \phi\in \ell^p G \}$, a dense $G$-invariant $*$-subalgebra of $C(X)$. We denote by $A_p(X)\rtimes G$ the Banach $*$-algebra obtained by completing $A_p(X)\rtimes_{\textup{alg}} G$ under the norm $\|a\|_{(p)}:=\|a\|+\|[P_{\ell^2G}, \lambda_\mu(a)]\|_{\LL^p}$. Then $A_p(X)\rtimes G$ is a dense subalgebra of $C(X)\rcross G$ which is closed under holomorphic functional calculus. 

Let $n$ be an odd integer such that $n+1>p$. The Connes' Chern character of the $p$-summable Fredholm module $(\lambda_\mu, P_{\ell^2G})$ is represented by the following cyclic $n$-cocycle on $A_p(X)\rtimes G$:
\[ \Phi (a^0, a^1,\ldots , a^n) 
:= \Tr \Big((2P_{\ell^2G}-1)\big[P_{\ell^2G}, \lambda_\mu(a^0)\big]\big[P_{\ell^2G}, \lambda_{\mu}(a^1)\big] \cdots \big[P_{\ell^2G}, \lambda_{\mu}(a^n)\big]\Big)\]
Next, we give a more explicit formula for this cocycle. Clearly, it suffices to understand the case of basic elements $a^i=\phi_i g_i$, where $\phi_i\in A_p(X)$ and $g_i\in G$. To state our result, we define yet another dynamical version of a commonly used probabilistic notion. The \emph{$G$-covariance} of $\phi, \psi\in C(X)$ is the map $\cov_G(\phi,\psi):G\to\C$ given as follows:
\begin{align*}
\cov_G (\phi,\psi)=\E_G\big(\phi \psi^*\big)-\E_G(\phi)\E_G(\psi^*)
\end{align*}

\begin{prop}
Let $\phi_0, \phi_1, \dots , \phi_n\in A_p(X)$ and $g_0, g_1, \dots , g_n\in G$. If $g_0g_1\dots g_n \neq 1\in G$ then $\Phi(\phi_0g_0, \phi_1g_1, \dots , \phi_ng_n)=0$. If $g_0g_1\dots g_n = 1\in G$ then 
\begin{align*}
(-1)^\frac{n+1}{2}\Phi&(\phi_0g_0, \phi_1g_1, \dots , \phi_ng_n)
\\&=\sum_{h\in G}\Big(\cov_G(\psi_0,\psi_1)\: \cov_G(\psi_2, \psi_3)\: \dots \:\cov_G(\psi_{n-1},\psi_n)\Big)(h)\\
&-\sum_{h\in G}\Big(\cov_G(\psi_n,\psi_0)\:\cov_G(\psi_1, \psi_2)\:\dots \:\cov_G(\psi_{n-2},\psi_{n-1}) \Big)(h),
\end{align*}
where $\psi_0:=\phi_0$, $\psi_1:=g_0.\phi_1$, $\psi_2:=g_0g_1.\phi_2$, $\dots$, $\psi_n:=g_0g_1\dots g_{n-1}.\phi_n$.
\end{prop}

\begin{proof}
Since $P_{\ell^2G}$ commutes with $\lambda_\mu$, we have $\lambda_\mu(g)\big[P_{\ell^2G}, \lambda_{\mu}(\phi h)\big]=\big[P_{\ell^2G}, \lambda_{\mu}(g.\phi)\big]\lambda_{\mu}(gh)$ for $g,h\in G$. Therefore
 \[\Phi(\phi_0g_0, \phi_1g_1,\dots , \phi_ng_n)=\Phi(\psi_0,\psi_1, \dots , \psi_n g)\]
 where $g:=g_0g_1\dots g_n$. Recall now the notation $\Pi(\cdot)$ from the proof of Proposition~\ref{from deviation to Fredholm}. As $[P_{\ell^2G}, \lambda_\mu(\psi)]=\Pi(\psi)^*-\Pi(\psi^*)$ and $\Pi(\psi)\Pi(\psi')=0$, we compute
\begin{align*}
(2P_{\ell^2G}-1)&\big[P_{\ell^2G}, \lambda_{\mu}(\psi_0)\big]\big[P_{\ell^2G}, \lambda_{\mu}(\psi_1)\big] \cdots \big[P_{\ell^2G}, \lambda_{\mu}(\psi_n)\big]\\ 
&= (-1)^\frac{n+1}{2} \Pi(\psi_0)^*\Pi(\psi_1^{*})\Pi(\psi_2)^*\dots \Pi(\psi_n^{*})-(-1)^\frac{n+1}{2}\Pi(\psi_0^{*})\Pi(\psi_1)^*\Pi(\psi_2^{*})\dots\Pi(\psi_n)^*
\end{align*}
Therefore
\begin{align*}
(-1)&^\frac{n+1}{2}\Phi(\psi_0,\psi_1, \dots , \psi_ng)\\
&=\Tr \Big(\Pi(\psi_0)^*\Pi(\psi_1^{*})\Pi(\psi_2)^*\dots \Pi(\psi_n^{*})\lambda_\mu(g)\Big)-\Tr \Big(\Pi(\psi_0^{*})\Pi(\psi_1)^*\Pi(\psi_2^{*})\dots\Pi(\psi_n)^*\lambda_{\mu}(g)\Big) 
\end{align*}
since $\Pi(\psi)$ is a $p$-summable operator whenever $\psi\in A_p(X)$. Using the fact that $\Pi(\psi)^*\Pi(\psi')$ is the multiplication operator $\M(\cov_G(\psi,\psi'))$ on $\ell^2G$, we have
\begin{align*}
\Tr &\Big(\Pi(\psi_0)^*\Pi(\psi_1^{*})\Pi(\psi_2)^*\dots \Pi(\psi_n^{*})\lambda_\mu(g)\Big)\\
&=\Tr \bigg(\M\Big(\cov_G(\psi_0,\psi_1^{*})\:\cov_G(\psi_2, \psi_3^*)\:\dots \:\cov_G(\psi_{n-1},\psi_n^{*})\Big)\lambda_{\mu}(g)\bigg);
\end{align*}
on the other hand, as $\Pi(\psi_n)^*\lambda_\mu(g)=\lambda_\mu(g)\Pi(g^{-1}\psi_n)^*$, we can write:
\begin{align*}
\Tr& \Big(\Pi(\psi_0^{*})\Pi(\psi_1)^*\Pi(\psi_2^{*})\dots\Pi(\psi_n)^*\lambda_{\mu}(g)\Big)\\
&= \Tr \Big(\Pi(\psi_0^{*})\Pi(\psi_1)^*\dots \Pi(\psi_{n-1}^{*}) \lambda_{\mu}(g)\Pi(g^{-1}\psi_n)^*\Big)
\\&= \Tr \Big(\Pi(g^{-1}\psi_n)^*\Pi(\psi_0^{*})\Pi(\psi_1)^*\dots \Pi(\psi_{n-1}^{*}) \lambda_{\mu}(g)\Big)\\
&= \Tr \bigg(\M\Big(\cov_G(g^{-1}\psi_n,\psi_0^{*})\:\cov_G(\psi_1, \psi_2^*)\:\dots\: \cov_G(\psi_{n-2},\psi_{n-1}^{*})\Big)\lambda_{\mu}(g)\bigg)
\end{align*}
Now $\Tr\big((\M\theta) \lambda_\mu(g)\big)$ equals $0$ if $g\neq 1$, respectively $\sum_{h\in G} \theta(h)$ if $g=1$. Thus, for $g\neq 1$ we get $\Phi(\psi_0,\psi_1, \dots , \psi_ng)=0$, and for $g= 1$ we get the claimed formula for $(-1)^\frac{n+1}{2}\Phi(\psi_0,\psi_1, \dots , \psi_ng)$.\end{proof}


\section{Metric-measure structure on the boundary of a hyperbolic space}
Our aim is to realize the paradigm described in the previous section in the case of a non-elementary hyperbolic group acting on its boundary. In this section, we prepare the ground by discussing the metric-measure structure on the boundary of a hyperbolic space in the sense of Gromov \cite{Gromov}. In \S\ref{bhs} we recall some basic facts on hyperbolic spaces and their boundaries. In \S\ref{vm} we focus on the family of visual metrics, and their induced Hausdorff measures, on the boundary of a hyperbolic space. Our main references for the (mostly standard) material in these first two subsections are Ghys - de la Harpe \cite{GdH} and V\"ais\"al\"a \cite[Section 5]{Vai}. In \S\ref{gga} we describe results of Coornaert \cite{Coo} on Hausdorff dimensions and Hausdorff measures for visual metrics, in the case when the hyperbolic space carries a geometric group action. Finally, in \S\ref{va} we discuss the notion of visual dimension.

\subsection{The boundary of a hyperbolic space}\label{bhs} Let $(X,d)$ be a proper geodesic space. The \emph{Gromov product} of  $x,y\in X$ with respect to $p\in X$ is defined by the formula
\[(x,y)_p: =\tfrac{1}{2}\big(d(p,x)+d(p,y)-d(x,y)\big).\]
The space $X$ is \emph{$\delta$-hyperbolic}, where $\delta\geq 0$, if
\[(x,y)_p\geq \min\big\{(x,z)_p,(y,z)_p\big\}-\delta\]
for all $x,y,z,p\in X$.

 Let $X$ be $\delta$-hyperbolic and fix a basepoint $o\in X$. A sequence $(x_i)$ \emph{converges to infinity} if $(x_i,x_j)_o\to\infty$ as $i,j\to\infty$. Two sequences $(x_i)$, $(y_i)$ converging to infinity are \emph{asymptotic} if $(x_i,y_i)_o\to\infty$ as $i\to\infty$. The asymptotic relation is an equivalence on sequences converging to infinity. A basepoint change modifies the Gromov product by a uniformly bounded amount, so convergence to infinity and the asymptotic relation are independent of the chosen basepoint $o\in X$. The \emph{boundary} of $X$, denoted $\bd X$, is the set of asymptotic classes of sequences converging to infinity. A sequence $(x_i)\subseteq X$ \emph{converges to} $\xi\in\bd X$ if $(x_i)$ converges to infinity, and the asymptotic class of $(x_i)$ is $\xi$. 

The Gromov product on $\bd X\times \bd X$ is defined as follows:
\begin{align*}
(\xi,\xi')_o:=\inf \big\{\liminf\; (x_i,x'_i)_o: \; x_i\to\xi,\; x'_i\to\xi'\big\}
\end{align*}
If $\xi=\xi'$, then $(\xi,\xi')_o=\infty$. If $\xi\neq \xi'$, then the sequence $(x_i,x'_i)_o$ is bounded whenever $x_i\to\xi$ and $x'_i\to\xi'$, hence $(\xi,\xi')_o<\infty$. It turns out that 
\begin{align}\label{liminf limsup}
(\xi,\xi')_o\leq \liminf\; (x_i,x'_i)_o\leq \limsup\; (x_i,x'_i)_o \leq (\xi,\xi')_o+ 2\delta \qquad (x_i\to\xi, x'_i\to\xi').
\end{align}
Similarly, we define the Gromov product on $X\times \bd X$ by setting
\begin{align*}
(x,\xi)_o:=\inf \big\{\liminf\; (x,x_i)_o: \; x_i\to\xi\big\}.
\end{align*}
In this case we have:
\begin{align}\label{liminf limsup 2}
(x,\xi)_o\leq \liminf\; (x,x_i)_o\leq \limsup\; (x,x_i)_o \leq (x,\xi)_o+ \delta\qquad (x_i\to\xi) 
\end{align}

In particular, \eqref{liminf limsup} and \eqref{liminf limsup 2} show that one could take sup instead of inf, or $\limsup$ instead of $\liminf$, in the definition of the Gromov product on $\bd X\times\bd X$, respectively $X\times \bd X$; all these variations would be within $2\delta$, respectively $\delta$, of each other.

\subsection{Visual metrics}\label{vm} Equipped with a canonical topology defined in terms of the Gromov product, the boundary $\bd X$ is compact and metrizable (see \cite[Ch.7, \S2]{GdH}). But the metric structure on $\bd X$, which is of great importance in this paper, is a more subtle issue. 

\begin{defn} A \emph{visual metric} on $\bd X$ is a metric $d_\e$ satisfying $d_\e\asymp\exp(-\e(\cdot,\cdot)_o)$ for some $\e>0$, called the \emph{visual parameter} of $d_\e$.
\end{defn}

This definition is independent of the chosen basepoint $o\in X$, and every visual metric determines the canonical topology on $\bd X$. In general, there is no natural choice of visual metric on $\bd X$. We thus take the \emph{visual gauge} $\mathcal{V}(\bd X)$ consisting of all visual metrics on $\bd X$, and we accept $\mathcal{V}(\bd X)$ in its entirety as giving the metric structure on $\bd X$.

\begin{lem}[Scaling]\label{scaling lemma}
Let $d_\e$ and $d_{\e'}$ be two visual metrics. Then:
\begin{itemize}
\item[$\bullet$] $d_\e$ and $d_{\e'}$ are H\"older equivalent: $d^{1/\e}_{\e}\asymp_{\e,\e'} d^{1/\e'}_{\e'}$;
\item[$\bullet$] the corresponding Hausdorff dimensions are inversely proportional to the visual parameter: $\e \:\hdim(\bd X,d_\e)=\e' \:\hdim (\bd X,d_{\e'})$;
\item[$\bullet$] the corresponding Hausdorff measures are comparable: $H_\e \asymp _{\e,\e'} H_{\e'}$.
\end{itemize}
\end{lem}
A priori, Hausdorff dimensions and Hausdorff measures corresponding to visual metrics could degenerate to $0$ or $\infty$. In other words, a visual metric need not generate a meaningful measure-theoretic structure. As we shall see in \S\ref{gga}, a geometric group action on $X$ brings a remarkable measure-theoretic regularity to the visual structure of $\bd X$.

Visual metrics do exist, provided that the visual parameter is small with respect to $1/\delta$. Furthermore, there is a companion metric-like map on $X\times \bd X$, which is visual in the corresponding way:

\begin{lem}[Small visual range]\label{small visual range} Let $\e>0$ be such that $\e\delta<1/5$. Then:
\begin{itemize}
\item[$\bullet$] there exists a visual metric $d_\e$, having visual parameter $\e$;
\item[$\bullet$] there exists $d_\e :X\times \bd X\to [0,\infty)$ satisfying $d_\e(\cdot, \cdot) \asymp \exp(-\e(\cdot,\cdot)_o)$ on $X\times \bd X$, and 
\[|d_\e(x, \xi)-d_\e(x, \xi')|\leq d_\e(\xi, \xi') \leq d_\e(x, \xi)+d_\e(x, \xi')\qquad (x\in X;\xi,\xi'\in \bd X).\]

\end{itemize}
\end{lem}

The small range for visual parameters is by no means optimal, and we have to allow for the possibility that visual metrics may exist for parameters outside of the small range - which is, in fact, what we did by considering the visual gauge $\mathcal{V}(\bd X)$. Statements about visual metrics on $\bd X$ are sometimes proved by first dealing with visual parameters in the small range, and then extended using the Scaling Lemma~\ref{scaling lemma}.

\subsection{Geometric group actions}\label{gga} 
Now let $X$ be a hyperbolic space admitting a \emph{geometric} -- that is, isometric, proper and cocompact -- action of a group $\G$. Then $\G$ is hyperbolic and the boundary $\bd X$ serves as a topological model for $\bd \G$. In what follows, we assume that $\G$ is non-elementary, that is, $\G$ is neither finite, nor virtually infinite cyclic. In terms of the space $X$, the non-elementary hypothesis on $\G$ means that $\bd X$ is infinite as a set.

The action of $\G$ on $X$ extends to the boundary $\bd X$. We have $(gx,gx')_o\geq (x,x')_o-d(o,go)$ for all $g\in \G$ and $x,x'\in X$, which implies that 
\begin{align}\label{boundary action is Lipschitz}
(g\xi,g\xi')_o\geq (\xi,\xi')_o-d(o,go)
\end{align} 
for all $g\in \G$ and $\xi,\xi'\in \bd X$. Therefore $\G$ acts by Lipschitz maps on $(\bd X, d)$ for any choice of visual metric $d\in \mathcal{V}(\bd X)$.

The \emph{exponent} of the action of $\G$ on $X$ can be defined in two equivalent ways: as a growth exponent
\[e_{\G \act X}=\limsup_{R\to \infty}\Big( \frac{1}{R}\ln \big|\{g\in\G: d(o,go)\leq R\}\big|\Big),\]
or as a critical exponent
\[e_{\G \act X}=\inf \Big\{s \: :\: \sum_{g\in \G}\exp(-sd(o,go))<\infty\Big\}.\]
The exponent is finite and non-zero, and its definition does not depend on the basepoint $o\in X$. Furthermore, it does not depend on the group $\G$ acting geometrically on $X$. Viewed as an intrinsic characteristic of the space $X$, the exponent $e_{\G \act X}$ becomes the volume entropy.

The Patterson - Sullivan theory developed by Coornaert in \cite{Coo} plays a crucial role in understanding the growth of $\G$-orbits in $X$, and the Hausdorff dimensions and measures associated to visual metrics on $\bd X$. In what concerns the orbit growth, we have the following exponential asymptotics (\cite[Thm.7.2]{Coo}):
\begin{align}\label{growth} 
\big|\{g\in\G: d(o,go)\leq R\}\big|\asymp _o \exp(e_{\G\act X} R)
\end{align}

As for the visual structure of $\bd X$, we have the following important fact:

\begin{lem}\label{Ahlfors} 
Endow the boundary $\bd X$ with a visual metric $d_\e$. Then the Hausdorff dimension $\hdim(\bd X, d_\e)$ equals $e_{\G\act X}/\e$, and the normalized Hausdorff measure $\mu_\e$ satisfies
\begin{align}\label{polgrowth}
\mu_\e(B_r)\asymp r^{\hdim(\bd X, d_\e)}
\end{align} 
for all closed balls $B_r$ of radius $0\leq r\leq\mathrm{diam} (\bd X, d_\e)$.
\end{lem}

For sufficiently small visual parameters $\e$, this follows from \cite[Prop.7.4, Cor.7.5, Cor.7.6]{Coo}; to get Fact~\ref{Ahlfors} for arbitrary visual parameters, use Lemma~\ref{scaling lemma}. The polynomial growth of balls stipulated in \eqref{polgrowth} says that $(\bd X, d_\e)$ an \emph{Ahlfors regular} metric space (see \cite[pp.60-62]{Hei}).

\subsection{Visual dimension}\label{va} Let $X$ be a hyperbolic space. We cannot really assign a Hausdorff dimension to the boundary $\bd X$, since there is no canonical choice of metric on $\bd X$. Instead, we work with the following:

\begin{defn}
The \emph{visual dimension} of $\bd X$, denoted $\visdim \bd X$, is the infimal Hausdorff dimension of $(\bd X, d)$ as $d$ runs over the visual metrics on $\bd X$.
\end{defn}

The topological dimension of a metric space is no greater than the Hausdorff dimension with respect to a compatible metric; thus $\visdim \bd X\geq \topdim \bd X$. Actually, we have the following chain of inequalities:
\begin{align*}
\visdim \bd X\geq \textrm{A-confdim}\: \bd X\geq \textrm{confdim}\: \bd X\geq \topdim \bd X
\end{align*}
The \emph{conformal dimension} of $\bd X$, denoted $\textrm{confdim}\: \bd X$, is a notion of metric dimension which only depends on the quasi-isometry type of $X$. It resolves the metric ambiguity at the boundary by taking all possible metrics which are equivalent to a visual metric in a suitable sense. The original definition, due to Pansu, uses quasi-conformal equivalence; more recently, the closely related quasi-M\"obius (equivalently, quasi-symmetric) equivalence seems to be favored. Then $\textrm{confdim}\: \bd X$ is defined as the infimal Hausdorff dimension of $(\bd X, d)$ as $d$ runs over all metrics  which are equivalent to a visual metric. See \cite[Section 14]{KB} for an overview of the quasi-conformal version of $\textrm{confdim}\: \bd X$.

From a measure-theoretic point of view, the equivalence relation used for defining the conformal dimension is too loose. For hyperbolic spaces admitting geometric group actions, the notion of \emph{Ahlfors conformal dimension} strikes a compromise by restricting the equivalence relation to Ahlfors regular metrics. Namely, $\textrm{A-confdim}\: \bd X$ is defined as the infimal Hausdorff dimension of $(\bd X, d)$ as $d$ runs over all Ahlfors regular metrics on $\bd X$ which are quasi-M\"obius equivalent to a visual metric. The Ahlfors conformal dimension is a key concept for much of the current work on boundaries of hyperbolic spaces from the perspective of analysis on metric spaces. See \cite{Kle} for more on these matters.

\begin{rem}
If $X$ admits a geometric group action, then Lemma~\ref{Ahlfors} shows that the visual dimension measures the range of visual parameters:
\begin{align*}
\visdim \bd X=\frac{\textrm{volume entropy of } X}{\textrm{vispar}\:\bd X}
\end{align*}
where $\textrm{vispar}\: \bd X=\sup \{\e>0 : \e \textrm{ is the parameter of a visual metric on }\bd X\}$. The point is that one should think of $\textrm{vispar}\: \bd X$ as a way to encode the negative curvature of $X$. The simplest manifestation of this idea is the following fact: if $X$ is a CAT($\kappa$) space, where $\kappa<0$, then $\sqrt{-\kappa}$ is a visual parameter so $\textrm{vispar}\: \bd X\geq \sqrt{-\kappa}$. A coarse version of this fact was investigated by Bonk and Foertsch in \cite{BF}. They define a notion of \emph{asymptotic upper curvature} for hyperbolic spaces which is invariant under rough isometries, and which agrees with the metric notion of curvature: if $X$ is a CAT($\kappa$) space, then $X$ has asymptotic upper curvature $K_u(X)\leq \kappa$. Bonk and Foertsch go on to show that $\textrm{vispar}\: \bd X=\sqrt{-K_u(X)}$ for a hyperbolic space $X$.
\end{rem}

\begin{rem}
We do not know whether the visual dimension of the boundary is a quasi-isometry invariant. Even more pertinent for this paper would be to know whether the visual dimension of the boundary is an invariant for hyperbolic spaces carrying a geometric action of a given group.
\end{rem}

\section{Finitely summable odd Fredholm modules for $\hi$}\label{sec: finite summability}
We are ready to realize the paradigm described in Section \ref{paradigm}. Let $\G$ be a non-elementary hyperbolic group, and let $(X,d)$ be a hyperbolic space carrying a geometric action of $\G$. Pick a visual metric $d_\e\in\mathcal{V}(\bd X)$, and let $\mu_\e$ denote the corresponding normalized Hausdorff measure on $\bd X$.

Let $\mathrm{H\ddot{o}l}_\alpha(\bd X, d_\e)$ be the algebra of complex-valued $\alpha$-H\"older functions, where $\alpha>0$. Then $\mathrm{H\ddot{o}l}_\alpha(\bd X, d_\e)$ is $\G$-invariant, as $\G$ acts by Lipschitz maps on $(\bd X, d_\e)$. When $\alpha\leq 1$, $\mathrm{H\ddot{o}l}_\alpha(\bd X, d_\e)$ contains the algebra of Lipschitz functions $\Lip(\bd X, d_\e)$, in particular it is dense in $C(\bd X)$.

Recalling \eqref{iint}, we have
\begin{align*}\dev_\G \phi(g)\leq\|\phi\|_{\mathrm{H\ddot{o}l}_\alpha}\;\sqrt{\tfrac{1}{2}\iint d_\e(g\xi,g\xi')^{2\alpha} \:\:d\mu_\e(\xi) d\mu_\e(\xi')}.
\end{align*}
for each $\phi\in \mathrm{H\ddot{o}l}_\alpha(\bd X, d_\e)$. This leads us to the following

\begin{lem}\label{double integral} Let $\alpha>0$ and pick $o\in X$. 
\begin{itemize}
\item[i)]
There exists $C>0$ such that, for all $g\in\G$, we have
\[\iint d_\e(g\xi,g\xi')^{2\alpha}\;d\mu_\e(\xi) d\mu_\e(\xi')\geq C \exp(-2\alpha \e\: d(o,go)).\]
\item[ii)]
Assume $\alpha<\tfrac{1}{2}\hdim (\bd X,d_\e)$. Then there exists $C'>0$ such that, for all $g\in\G$, we have
\[\iint d_\e(g\xi,g\xi')^{2\alpha}\;d\mu_\e(\xi) d\mu_\e(\xi')\leq C' \exp(-2\alpha \e\: d(o,go)).\]
\end{itemize}
\end{lem}

The important part of the lemma is, of course, ii): it implies that, for $0<\alpha<\tfrac{1}{2}\hdim (\bd X,d_\e)$, there is some constant $C''>0$ such that
\begin{align}\label{deviation estimate via length}
\dev_\G \phi(g)\leq C'' \|\phi\|_{\mathrm{H\ddot{o}l}_\alpha} \exp(-\alpha \e\: d(o,go))
\end{align}
for all $\phi\in \mathrm{H\ddot{o}l}_\alpha(\bd X, d_\e)$ and for all $g\in\G$. The purpose of i) is to show that we are getting the correct asymptotics in ii).

\begin{proof} i) is straightforward. For $\xi,\xi'\in\bd X$, we have $(g\xi,g\xi')_o\leq d(o,go)+(\xi,\xi')_o$ by \eqref{boundary action is Lipschitz}. Therefore $d_\e(g\xi,g\xi')\geq c\exp(-\e\: d(o,go))\:d_\e(\xi,\xi')$ for some $c>0$, which implies that
\[\iint d_\e(g\xi,g\xi')^{2\alpha}\;d\mu_\e(\xi) d\mu_\e(\xi') \geq \bigg(c^{2\alpha} \iint d_\e(\xi,\xi')^{2\alpha}\;d\mu_\e(\xi) d\mu_\e(\xi') \bigg) \exp(-2\alpha \e\: d(o,go)).\]
ii) is more involved. We start by assuming that the visual parameter $\e$ is in the small visual range, so that Lemma~\ref{small visual range} applies. Let $\xi, \xi'\in\bd X$, and let $x_i\to\xi, x_i'\to\xi'$ with $x_i,x_i'\in X$. One checks that
\[(gx_i,gx'_i)_o+(g^{-1}o,x_i)_o+(g^{-1}o,x'_i)_o\geq d(o,go);\] 
taking $\limsup$ and using \eqref{liminf limsup} and \eqref{liminf limsup 2}, we obtain that
\[(g\xi,g\xi')_o+(g^{-1}o,\xi)_o+(g^{-1}o,\xi')_o\geq d(o,go)-4\delta.\]
Thus there is $C_1>0$ such that
\begin{align}\label{annular bound}
d_\e(g\xi,g\xi')\: d_\e(g^{-1}o,\xi)\: d_\e(g^{-1}o,\xi')\leq C_1 \exp(-\e\: d(o,go))
\end{align}
For each integer $k\geq 1$ we put
\[\Delta_k= \big\{\xi\in \bd X: \exp(-\e k)\leq d_\e(g^{-1}o,\xi)\leq \exp(-\e (k-1))\big\}.\]
Then $\Delta _k$ is measurable, since $d_\e(g^{-1}o,\cdot)$ is continuous on $\bd X$. (As an aside remark, we mention that $d_\e(g^{-1}o,\cdot)$ is bounded below by a constant multiple of $\exp(-\e\: d(o,go))$, so $\Delta_k$ is in fact empty for $k\gg d(o,go)$.) If $\xi, \xi'$ are in $\Delta_k$, then 
\[d_\e(\xi,\xi')\leq d_\e(g^{-1}o,\xi)+d_\e(g^{-1}o,\xi')\leq 2 \exp(-\e (k-1))=(2e^\e) \exp(-\e k)\] 
and it follows from Lemma~\ref{Ahlfors} that there is $C_2>0$, not depending on $k$, such that
\begin{align}\label{annular measure}
\mu_\e(\Delta_k)\leq C_2\big(\exp(-\e k)\big)^{e_{\G\act X}/\e}= C_2 \exp(-e_{\G\act X}\:k).
\end{align}
Now \eqref{annular bound} and \eqref{annular measure} enable us to estimate as follows:
\begin{align*}
 \int_{\Delta_k}\int_{\Delta_{k'}} d_\e(g\xi,g\xi')^{2\alpha}\;d\mu_\e(\xi) d\mu_\e(\xi') 
&\leq C_3 \exp(-2\alpha\e\: d(o,go)) \exp(2\alpha \e (k+k')) \mu_\e(\Delta_k)\mu_\e(\Delta_{k'})\\
&\leq C_4  \exp(-2\alpha\e\: d(o,go)) \exp((2\alpha\e-e_{\G\act X})(k+k'))
\end{align*}
Recall, we are assuming $\alpha<\tfrac{1}{2}\hdim (\bd X,d_\e)$, that is, $2\alpha\e<e_{\G\act X}$. We conclude that
\begin{align*}
\iint d_\e(g\xi,g\xi')^{2\alpha}\;d\mu_\e(\xi) d\mu_\e(\xi') &\leq  \sum_{ k,k'\geq 0} \int_{\Delta_k}\int_{\Delta_{k'}} d_\e(g\xi,g\xi')^{2\alpha}\;d\mu_\e(\xi) d\mu_\e(\xi') \\
&\leq C_4 \exp(-2\alpha\e\: d(o,go)) \sum_{ k,k'\geq 0} \exp((2\alpha\e-e_{\G\act X})(k+k'))\\
&\leq C_5 \exp(-2\alpha\e\: d(o,go)).
\end{align*}
Now let $\e$ be an arbitrary visual parameter. Pick $\e_0$ in the small visual range, and let $d_{\e_0}$ be a corresponding visual metric. By Lemma~\ref{scaling lemma}, we have
\[\iint d_\e(g\xi,g\xi')^{2\alpha}\;d\mu_\e(\xi) d\mu_\e(\xi')\asymp \iint d_{\e_0}(g\xi,g\xi')^{2\alpha\e/\e_0}\;d\mu_{\e_0}(\xi) d\mu_{\e_0}(\xi').\]
As $\alpha\e/\e_0<\tfrac{1}{2}\hdim (\bd X,d_{\e_0})$, the first part of the proof shows that the right hand side is bounded from above by a constant multiple of $\exp(-2(\alpha\e/\e_0)\e_0\: d(o,go))=\exp(-2\alpha\e\: d(o,go))$. \end{proof}

With Lemma \ref{double integral} at hand, we may prove the main theorem of this section:

\begin{thm}\label{sharp general} Let $D:=\hdim (\bd X, d_\e)$. Then $(\lambda_{\mu_\e}, P_{\ell^2\G})$ is a Fredholm module for $\hi$ which is $p$-summable for every $p>\max\{2,D\}$. Furthermore, if $D>2$ then $(\lambda_{\mu_\e}, P_{\ell^2\G})$ is $D^+$-summable.
\end{thm}

\begin{proof} If $p>\max\{2,D\}$ then we may choose $\alpha$ such that
\[\alpha<\tfrac{1}{2}D, \qquad p\alpha >D, \qquad \alpha\leq 1.\]
The first bound on $\alpha$ allows us to employ \eqref{deviation estimate via length}; the second bound says that $p\alpha\e>D\e=e_{\G\act X}$ so the right-hand side of \eqref{deviation estimate via length} is in $\ell^p\G$; finally, the third bound guarantees that $\mathrm{H\ddot{o}l}_\alpha(\bd X, d_\e)$ is dense in $C(\bd X)$. Therefore $\mu_\e$ has $\ell^p\G$-deviation, and we conclude that $(\lambda_{\mu_\e}, P_{\ell^2\G})$ is $p$-summable.

Assume now that $D>2$. Using $\alpha=1$ in \eqref{deviation estimate via length}, we have
$\dev_\G \phi(g)\leq C \|\phi\|_\Lip \exp(-\e d(o,go))$ for all $g\in \G$ and $\phi\in\Lip(\bd X, d_\e)$. Let $T$ denote the multiplication operator by $g\mapsto \exp(-\e d(o,go))$ on $\ell^2\G$. We claim that $T\in \mathcal{L}^{D+} (\ell^2\G)$. Once we know this, it follows that multiplication by $\sigma_\G \phi$ is in $\mathcal{L}^{D+} (\ell^2\G)$ for all $\phi\in\Lip(\bd X, d_\e)$, and the proof of Proposition~\ref{from deviation to Fredholm} shows that $(\lambda_{\mu_\e}, P_{\ell^2\G})$ is a $D^+$-summable Fredholm module. In order to prove our claim that $\mu_n(T)=O(n^{-1/D})$, we first control a subsequence of singular values for $T$. Let $m_k=\big|\{g\in \G: d(o,go)\leq k\}\big|$; thus $m_k\asymp \exp(e_{\G\act X} k)$ by \eqref{growth}. We have
\begin{align*}
\mu_{m_k+1}(T) < \exp(-\e k)= \exp(e_{\G\act X} k)^{-1/D}\leq C_1\: m_k^{-1/D}.
\end{align*}
For an arbitrary positive integer $n$, let $k$ be such that $m_k+1\leq n\leq m_{k+1}+1$. Then
\[\mu_{n}(T)\leq \mu_{m_k+1}(T)\leq C_1\: m_k^{-1/D}\leq C_1\: n^{-1/D} \Big(\frac{m_{k+1}+1}{m_k}\Big)^{1/D}\leq C_2\: n^{-1/D}\] 
for some constant $C_2$ independent of $n$ and $k$. \end{proof}

We remark that $D>2$ whenever $\bd X$ has topological dimension greater than $2$. Let us also remind the reader that the condition $p>2$ is more of a structural obstruction rather than a technical artifact. On the one hand, Proposition~\ref{from deviation to Fredholm} says that the summability of a regular Fredholm module is largely determined by the decay of the deviation. On the other hand, Proposition~\ref{2ergo} says, roughly speaking, that the decay of the deviation has to be faster than $\ell^2$ in the presence of double ergodicity. And this is the case for visual probability measures: they are known to be doubly ergodic.

Recall now that different choices of visual metrics on $\bd X$ yield comparable Hausdorff measures (Lemma~\ref{scaling lemma}), and that comparable measures have the same type of $\G$-deviation (Proposition~\ref{deviation for comparable measures}). We obtain

\begin{cor}\label{uniform summability}
 $(\lambda_{\mu_\e}, P_{\ell^2\G})$ is $p$-summable for every $p>\max\{2,\visdim \bd X\}$. Furthermore, if $\visdim \bd X >2$ and $\visdim\bd X$ is attained, then $(\lambda_{\mu_\e}, P_{\ell^2\G})$ is $(\visdim \bd X)^+$-summable.
\end{cor}

\begin{ex} Consider the free group $F_n$, where $n\geq 2$. We let $F_n$ act on its standard Cayley graph, the $2n$-valent regular tree $\mathcal{T}_{2n}$. As $\mathcal{T}_{2n}$ is $0$-hyperbolic, each $\e>0$ can be used as a visual parameter. Consequently, $\visdim \bd\mathcal{T}_{2n}=0$ but there is no visual metric for which the visual dimension is attained. The canonical visual metric of parameter $\e>0$ is $\exp(-\e(\cdot, \cdot))$, where the Gromov product is based at the identity element of $F_n$. All these canonical visual metrics determine the same Hausdorff measure on $\bd\mathcal{T}_{2n}$. After normalizing this canonical measure, we obtain a Fredholm module for $C(\bd F_n)\rtimes F_n$ which is $p$-summablefor every $p>2$.
\end{ex}

\begin{ex} Let $\G$ be a cocompact lattice in $\mathrm{Isom}(\mathbb{H}^{n})$. The boundary of $\mathbb{H}^{n}$ is the sphere $S^{n-1}$. Each $\e\in (0,1]$ can be used as a visual parameter, and the usual spherical metric is a visual metric with visual parameter $\e=1$. The corresponding normalized Hausdorff measure is the spherical measure $\sigma$, with Hausdorff dimension $n-1$. We thus have a Fredholm module for $C(S^{n-1})\rtimes \G$ which is $(n-1)^+$-summable when $n\geq 4$, and $p$-summable for every $p>2$ when $n=2,3$.

The Lipschitz functions with respect to a visual metric with parameter $\e$ are the $\e$-H\"older functions with respect to the spherical metric. But the former are dense in $C(S^{n-1})$, whereas the latter ones are the constant functions if $\e> 1$. Thus, the range of visual parameters on $S^{n-1}$ is $(0,1]$. It follows that $\visdim \bd\mathbb{H}^{n}=n-1$, and it is attained by the spherical metric.
\end{ex}

\begin{ex} More generally, let $X$ be a non-compact symmetric space of rank $1$: $X=\mathbb{H}^n_K$ where $K=\R, \C, \mathbb{H}$ and $n\geq 2$, or $K=\mathbb{O}$ and $n=2$. Put $k=\dim_\R K\in \{1,2,4,8\}$. Topologically, the boundary $\bd X$ is a sphere of dimension $nk-1$. The standard metric on $\bd X$, the so-called Carnot metric, is a visual metric with visual parameter $\e=1$. As the Carnot metric is geodesic, no $\e>1$ can be used as a visual parameter. Therefore $\visdim \bd X=\hdim\, \bd X$, the Hausdorff dimension of $\bd X$ equipped with the Carnot metric. The Mitchell - Pansu formula says that $\hdim\, \bd X=nk+k-2=\topdim \bd X+k-1$; in particular, $\hdim\, \bd X>2$ unless $X$ is the $2$- or $3$-dimensional real hyperbolic space. Hence, if $X\neq \mathbb{H}^2_\R, \mathbb{H}^3_\R$ and $\G$ is a cocompact lattice in the isometry group of $X$, then we get a $(\hdim\, \bd X)^+$-summable Fredholm module for $C(\bd X)\rtimes \G$.
\end{ex}


\section{The boundary extension class}\label{sec: boundary extension class}
By applying Proposition~\ref{independence of comparable measures}, we know the following: for a given topological realization of $\bd \G$ as the boundary of a hyperbolic space carrying a geometric action of $\G$, the regular Fredholm modules coming from visual probability measures determine the same class in $\K^1(\hi)$. In this section, we show that this class is independent of the topological realization of $\bd\G$, and we identify this class as the class determined by a certain boundary extension for $\G$.

\subsection{The boundary extension}\label{what is the boundary extension} Let $\overline{\G}=\G\cup \bd\G$ be the boundary compactification of $\G$. From the exact sequence of $\G$-$\Cstar$-algebras $0\Too C_0(\G)\Too C(\overline{\G})\Too C(\bd \G)\Too 0$ we obtain an exact sequence of $\Cstar$-crossed products by $\G$:
\begin{align}
\label{eq:crossed_extension}
0\Too C_0(\G)\rtimes \G\Too C(\overline{\G})\rtimes\G\Too C(\bd \G)\rtimes\G\Too 0
\end{align}
Each $\Cstar$-algebra in \eqref{eq:crossed_extension} is nuclear; in particular, the full and the reduced crossed products agree. The faithful representation of $C(\overline{\G})$ on $\ell^2\G$ by multiplication induces a faithful representation $C(\overline{\G})\rtimes\G\to \B(\ell^2\G)$. Under this representation, there is a canonical isomorphism between the ideal term $C_0(\G)\rtimes \G$ and the compact operators $\Comp (\ell^2\G)$. Hence \eqref{eq:crossed_extension} defines a class in the Brown--Douglas--Filmore group $\mathrm{Ext}(\hi)$. The nuclearity of $\hi$ implies that $\mathrm{Ext}(\hi)$ and $\K^1(\hi)$ are isomorphic. Concretely, this isomorphism is implemented as follows: a completely positive section for \eqref{eq:crossed_extension} determines, by the Stinespring construction, an odd Fredholm module for $\hi$ whose $\K^1$-class is independent of the choice of section. 
 
 \begin{defn}
\label{def:boundary_class}
The \emph{boundary extension class} $[\bd_\Gamma] \in \K^1(\hi)$ is the class defined by the extension \eqref{eq:crossed_extension}. 
\end{defn}

We shall see that each visual probability measure on $\bd\G$ determines a completely positive section for \eqref{eq:crossed_extension} which resembles the Poisson transform. The odd Fredholm modules corresponding to this section are the Fredholm modules we have been discussing.

The non-vanishing of the boundary extension class $[\bd_\G]$ in $\K^1(\hi)$ is a crucial point. From work of Emerson \cite{Emerson} and Emerson--Meyer \cite{Emerson:Euler}, we understand this particularly well in the case when $\G$ is torsion-free.

\begin{prop}[Emerson, Emerson--Meyer]\label{order tf} Assume that $\G$ is torsion-free. Then $[\bd_\G]$ vanishes in $\K^1(\hi)$ if and only if $\chi(\G)=\pm 1$. Furthermore, $[\bd_\G]$ has infinite order in $\K^1(\hi)$ if and only if $\chi(\G)=0$. 
\end{prop}

Let us explain the facts behind this proposition; throughout this paragraph, $\G$ is torsion-free. 

The first ingredient is the Poincar\'e duality for $\hi$ proved in \cite{Emerson}. It implies that the K-theory and the K-homology of $\hi$ are canonically isomorphic, with a parity shift: $\K_*(\hi)\simeq \K^{*+1}(\hi)$. Most importantly, the boundary extension class $[\bd_\G]\in \K^1(\hi)$ turn out to be the Poincar\'e dual of the unit class $[1_{\hi}]\in \K_0(\hi)$. It should be mentioned here that the proof from \cite{Emerson} - though most likely not Poincar\'e duality itself - needs the following mild symmetry condition on the boundary: $\bd \G$ has a continuous self-map without fixed points.
The symmetry condition is satisfied whenever $\bd\G$ is a topological sphere, a Cantor set or a Menger compactum. We are not aware of any example where the symmetry condition fails. The ``topologically rigid'' hyperbolic groups of Kapovich and Kleiner \cite{KK} come close, though, for their boundaries admit no self-homeomorphisms without fixed points.

The second ingredient comes from \cite{Emerson:Euler}, where a Gysin sequence for the inclusion $i:\Cred\G\to\hi$ is obtained. For $\G$ torsion-free, the Gysin sequence reads as follows:
\begin{align*}
0\Too \big\langle \chi(\G)\cdot[1_{\Cred\G}]\big\rangle & \Too  \K_0(\Cred\G)\stackrel{i_*}{\Too} \K_0(\hi)\Too \K^1(B\G)\Too 0\\
0 &\Too \K_1(\Cred\G)\stackrel{i_*}{\Too} \K_1(\hi)\Too \K^0(B\G)\Too  \big \langle \chi(\G)\big\rangle\Too 0
\end{align*}
Here $\chi(\G)$ is the Euler characteristic of $\G$, and $B\G$ is a finite, simplicial classifying space for $\G$ (for instance, a suitably large Vietoris - Rips complex). As a consequence, one knows that the order of the unit class $[1_{\hi}]\in \K_0(\hi)$ is $|\chi(\G)|$ when $\chi(\G)\neq 0$, respectively infinite when $\chi(\G)=0$. Additionally, it can be shown that $[\bd_\G]$ induces a nontrivial map \(\K_1(\hi) \to \Z\) when $\chi(\G)= 0$.

\subsection{A completely positive section for the boundary extension} Equip $\bd\G$ with a visual probability measure $\mu$. This means that $\mu$ is the normalized Hausdorff measure corresponding to a visual metric on $\bd X$, where $X$ is a hyperbolic space on which $\G$ acts geometrically.

As $(\lambda_{\mu}, P_{\ell^2\G})$ is a Fredholm module, we know that $g_*\mu$ accumulates to point measures in $\mathrm{Prob}(\bd \G)$ as $g\to\infty$ in $\G$ (Propositions~\ref{from deviation to Fredholm} and \ref{Furstenberg type condition}). Actually, this accumulation happens in a precise way:

\begin{lem}\label{sharp delta}  $g_*\mu\to \delta_\omega$ in $\mathrm{Prob}(\bd \G)$ as $g\to \omega\in \bd\G$ in $\overline \G$.
\end{lem}

\begin{proof} Let $(g_i)$ be a sequence in $\G$ converging in the compactification $\overline \G$ to $\omega\in \bd\G$. As above, let $X$ be a hyperbolic space on which $\G$ acts geometrically such that $\mu$ is a probability measure coming from a visual metric on $\bd X$. We have
\begin{align*}
\bigg| \int_{\bd X} \phi\: d(g_i)_*\mu-\phi(\omega)\bigg|=\bigg| \int_{\bd X} \phi(g_i\xi)-\phi(\omega)\; d\mu(\xi)\bigg| \leq \int_{\bd X} \big|\phi(g_i\xi)-\phi(\omega)\big|\; d\mu(\xi)
\end{align*}
for all $\phi \in C(\bd X)$. Without loss of generality we may take $\mu=\mu_\e$, where the visual parameter $\e$ is in the small visual range, so that Lemma~\ref{small visual range} applies.

Fix $\phi\in C(\bd X)$ and let $t>0$. The uniform continuity of $\phi$ provides us with some $R>0$ such that $|\phi(\xi)-\phi(\xi')|<t$ whenever $d_\e(\xi,\xi')<R$. The set
\begin{align*}
Z_i=\{\xi\in \bd X: d_\e(g_i\xi,\omega)\geq R\}
\end{align*}
is measurable, since $\xi\mapsto d_\e(g_i\xi,\omega)$ is continuous. We write:
\begin{align*}
\int_{\bd X} \big|\phi(g_i\xi)-\phi(\omega)\big|\; d\mu_\e(\xi) &= \int_{\bd X\setminus Z_i} \big|\phi(g_i\xi)-\phi(\omega)\big|\; d\mu_\e(\xi) + \int_{Z_i} \big|\phi(g_i\xi)-\phi(\omega)\big|\; d\mu_\e(\xi)\\
&\leq  t+2\|\phi\|_\infty\: \mu_\e (Z_i)
\end{align*}
Fix a basepoint $o\in X$. We first establish the following fact: there is $C_1>0$ such that
\begin{align}\label{split}
d_\e(g\xi,\omega)\:d_\e(g^{-1}o,\xi)\leq C_1\:d_\e(go,\omega)
\end{align}
for all $g\in \G$ and $\xi,\omega\in\bd\G$. Indeed, let $(x_i)$, respectively $(w_i)$, be sequences in $X$ converging to $\xi$, respectively $\omega$. One easily checks that
$(gx_i,w_i)_o+(g^{-1}o,x_i)_o\geq (go,w_i)_o$; taking $\limsup$ and then invoking \eqref{liminf limsup} and \eqref{liminf limsup 2}, we are led to $(g\xi,\omega)_o+(g^{-1}o,\xi)_o\geq (go,\omega)_o-3\delta$. Exponentiating this inequality, we get \eqref{split}.

Now if $\xi,\xi'\in Z_i$, then both $d_\e(g_i^{-1}o,\xi)$ and $d_\e(g_i^{-1}o,\xi')$ are at most $C_1R^{-1}\: d_\e(g_io,\omega)$ by \eqref{split}; hence $d_\e(\xi,\xi')\leq 2C_1R^{-1} d_\e(g_io,\omega)$. It follows from Lemma~\ref{Ahlfors} that there is $C_2>0$ such that
\begin{align*}
\mu_\e (Z_i)\leq C_2\: d_\e(g_io,\omega)^{e_{\G\act X}/\e}
\end{align*}
for each $i$. But $d_\e(g_io,\omega)\to 0$, since $g_i\to \omega$ in $\overline \G$, and we obtain 
\begin{align*}
\int_{\bd X} \big|\phi(g_i\xi)-\phi(\omega)\big|\; d\mu_\e(\xi)\leq 2t
\end{align*}
for $i$ large enough. The proof is complete.
\end{proof}

Using the $\G$-expectation with respect to $\mu$, we may extend continuous functions on the boundary $\bd\G$ to functions on the compactification $\overline{\G}$. Namely, for $\phi\in C(\bd\G)$, define $\Ebar(\phi)$ on $\overline{\G}$ by gluing $\phi$ to its $\G$-expectation $\E_\G(\phi)$: 
\begin{align*}
\Ebar(\phi)= 
\begin{cases} \phi & \text{on $\bd \G$}
\\
\E_\G (\phi) &\text{on $\G$}
\end{cases}
\end{align*}
By Lemma~\ref{sharp delta}, the extension $\Ebar (\phi)$ is continuous on $\overline{\G}$. Therefore the linear map
\[\Ebar: C(\bd \G)\to C(\overline{\G})\]
is a completely positive section for the quotient map $C(\overline{\G})\onto C(\bd \G)$ induced by restriction. We want to promote $\Ebar$ to a completely positive section for $C(\overline{\G})\rtimes \G\onto \hi$. Let
\begin{align}\label{def sec}
s : C(\bd \G)\rtimes _{\textrm{alg}} \G \to C(\overline{\G})\rtimes \G, \qquad s\big(\sum \phi_g\, g\big) = \sum \Ebar (\phi_g)\, g.
\end{align}

\begin{thm}\label{representing boundary class} With the above notations, we have the following:
\begin{itemize}
\item[i)] $s$ extends by continuity to a completely positive section for $C(\overline{\G})\rtimes \G\onto \hi$.

\item[ii)]  $(\lambda_{\mu}, P_{\ell^2\G})$ is a Fredholm module representing the boundary extension class $[\bd_\G]$.
\end{itemize}
\end{thm}

\begin{proof} Recall, $\lambda_{\mu}$ is the (faithful) regular representation of $\hi$ on $\ell^2(\G, L^2(\bd\G,\mu))$. Let $\pi$ denote the (faithful) representation of $C(\overline{\G})\rtimes\G$ on $\ell^2\G$. Formula \eqref{projection compressed lambda} says that $\pi\big(\Ebar(\phi)\big)= P_{\ell^2\G} \lambda_{\mu}(\phi) P_{\ell^2\G}$ for all $\phi \in C(\bd \G)$. On the other hand, we have $\pi(g)=P_{\ell^2\G} \lambda_{\mu}(g) P_{\ell^2\G}$ for all $g \in  \G$. It follows that
\begin{align}\label{stined}
\pi s(a)= P_{\ell^2\G} \lambda_{\mu}(a) P_{\ell^2\G}
\end{align}
for all $a\in C(\bd \G)\rtimes _{\textrm{alg}} \G$. Using \eqref{stined} and the faithfulness of $\pi$, we  see that $s$ extends by continuity to a completely positive section for $C(\overline{\G})\rtimes \G\onto \hi$. This concludes part i). The Stinespring dilation $\pi s= P_{\ell^2\G}\: \lambda_{\mu}\: P_{\ell^2\G}$ shows that ii) holds. \end{proof}
 
\begin{rem}
The completely positive section $s$ is $\G$-biequivariant: $s(gag')=gs(a)g'$ for all $g,g'\in \G$ and $a\in \hi$. The left $\G$-equivariance of $s$ follows from the $\G$-equivariance of $\Ebar$, namely $\Ebar(g.\phi)=g.\Ebar(\phi)$ for all $g\in \G$ and $\phi\in C(\bd \G)$. The right $\G$-equivariance of $s$ follows directly from the definition.
\end{rem}

We may pass from torsion-free to virtually torsion-free groups, and establish a version of Proposition~\ref{order tf} for this much larger class, by using Proposition ~\ref{finite index}. Note that the comparability condition, required in the statement of Proposition ~\ref{finite index}, is satisfied in our setting. Indeed, for any choice of visual metric on $\bd \G$, $\G$ acts by Lipschitz maps on $\bd\G$; consequently, if $\mu$ is a visual probability measure on $\bd\G$, then $g_*\mu$ is comparable to $\mu$ for all $g\in \G$. In light of Theorem~\ref{representing boundary class}, and of the fact that the boundary is preserved by passing to finite-index subgroups, Proposition ~\ref{finite index} implies the following:

\begin{prop}\label{finite index boundary} Let $\Lambda$ be a finite-index subgroup of $\G$. Then the restriction homomorphism $\mathrm{res}: \K^1(\hi)\to\K^1(C(\bd \Lambda)\rtimes \Lambda)$  sends $[\bd_\G]$ to $[\G:\Lambda] \cdot [\bd_{\Lambda}]$.
\end{prop}

Recall that the rational Euler characteristic of a virtually torsion-free group $\G$ is defined by the formula $\chi(\G)=\chi(\Lambda)/[\G: \Lambda]$ where $\Lambda$ is any torsion-free subgroup of finite index. Combining Proposition~\ref{order tf} and Proposition~\ref{finite index boundary}, we obtain:

\begin{cor}
Assume that $\G$ is virtually torsion-free. If $\chi(\G)\notin 1/\Z$ then $[\bd_\G]$ is not trivial. If $\chi(\G)=0$ then $[\bd_\G]$ has infinite order.
\end{cor}

The assumption of virtual torsion-freeness is a very mild one. A long-standing open problem asks whether all hyperbolic groups are virtually torsion-free.


\end{document}